\documentclass{birkjour}
 \newtheorem{theorem}{Theorem}[section]
 \newtheorem{corrolory}[theorem]{Corollary}
 \newtheorem{lemma}[theorem]{Lemma}
 
 \theoremstyle{definition}
 \newtheorem{definition}[theorem]{Definition}
 \theoremstyle{remark}
 
 \newtheorem{example}{Example}
 \numberwithin{equation}{section}
\usepackage{hyperref}
\usepackage{mathrsfs}
\usepackage{amsmath,geometry,amssymb}
\usepackage{graphicx}

\begin{document}

%
%
%
%
%
%
%
%
%

\title[]
 {Warped product pointwise bi-slant submanifolds of locally product Riemannian manifolds}


\author[Majeed]{Prince Majeed}

\address{%
	Department of Mathematics,\\National Institute of Technology Srinagar, \\
	190006, Kashmir, India.}

\email{prince$_{-}$05phd19@nitsri.net}

\author[Lone]{Mehraj Ahmad Lone}

\address{%
	Department of Mathematics,\\National Institute of Technology Srinagar, \\
	190006, Kashmir, India.}

\email{mehrajlone@nitsri.ac.in}





\begin{abstract}
In this paper we introduce the concept of pointwise bi-slant submanifolds of locally product Riemannian manifolds and studied warped product pointwise bi-slant submanifolds of locally product Riemannian manifolds.  We obtain some characterization results for warped products pointwise bi-slant submanifolds. Also, we provide some non-trivial examples of such warped product submanifolds.  
\end{abstract}

\maketitle

\section{\protect \bigskip Introduction} 	
\par In \cite{chen2005}, Chen introduced the  notion of slant submanifolds. It includes totally real as well as holomorphic submanifolds. Numerous geometer groups continue to study and conduct research on this idea of submanifolds. Recently, the related literature of slant submanifolds has been compiled in the form of two books by Chen, Shahid and Solamy (see \cite{chen a,chen b}). Since the introduction of slant submanifolds,  many generalizations and extensions of slant submanifolds have been introduced, like: semi-slant, pointwise slant, hemi-slant, pointwise hemislant and many more. The related literature of these kind of generalizations can be be found in (see, \cite{Garay,Etayo,VKhan,Lone,stonvic}).  A more generic class of submanifolds in the form of bi-slant submanifolds was introduced by Cabrerizo and Cariazo \cite{Cabrerizo}. This class of submanifolds acts as a natural generalization of CR, semi-slant,  slant,  hemi-slant submanifolds \cite{VKhan,Lone,papaghuic}. In connection to this generic notion of submanifolds, some of recent studies can be found in \cite{KON}. Further the extended notion of pointwise bi-slant submanifolds of Kaehler manifolds can be found in \cite{Uddin2}.\\

\par Bishop and O’Neill in $1960s$ introduced the concept of warped product manifolds. These manifolds find their applications both in physics as well as in mathematics.  Since then the study of warped product submanifolds has been investigated by many geometers (see, \cite{solamy,ch,chnn,Chnnn}).  In particular,  Chen started looking these warped products as submanifolds of different kinds of manifolds (see, \cite{5,6}). In this connection,  in Kaehlerian settings, he proved  besides CR- products the non-existence of warped products  of the form $N^\perp \times_f N^T$,  where $N^\perp,$ $N^T$  is a totally real  and holomorphic submanifold,  respectively \cite{Uddin}.  Now from the past two decades this area of research is an active area of research among many of the geometry groups.  For the overall development of the subject we refer the reader to see Chen's book on it \cite{Chenslant}.\\

\par Now while importing the survey of warped products to  slant cases,  Sahin  in \cite{B.sahin} proved the non-existence of  semi-slant warped products in any Kaehler manifold. Then in \cite{B.sahinn} he extended the study to pointwise semi-slant warped products of Kaeherian manifolds. Sahin \cite{Bb.Sahinnn} and Atcken \cite{A.Atcken} investigated warped product semi-slant submanifolds of locally product Riemannain manifolds. They proved there is no warped product semi-slant submanifold of the form $M_T \times_f M_\theta$ of a locally product Riemannian manifold $\bar{M}$ such that $M_T$ and $M_\theta$ are invariant and proper slant submanifolds of $\bar{M}$, respectively. Moreover they provided non-trival examples and proved a characterization  theorem for warped product semi-slant submanifolds of the form $M_\theta \times_f M_T$.\\
\par The main motivation for our paper is a recent study of Uddin, Alghamdi and Solamy \cite{JGP},  in which they have studied the geometry of warped product pointwise semi-slant submanifolds of locally product Riemannian manifolds. In this paper, we try to generalize the notion to bi-slant warped products of locally product Riemannain manifolds. We proved several results on pointwise bi-slant submanifolds of locally product Riemannain manifolds, in addition, we proved some characterization results for pointwise bi-slant submanifolds of locally product Riemannain manifolds. Later,  we also provide some non-trivial examples of such submanifolds. 
\section{Preliminaries}
Let $\bar{M}$ be an m-dimensional differential manifold with a tensor field $F$ of type (1,1) such that $F^2 = I$ and $F \neq \pm I$. Then we say that $\bar{M}$ is an almost product manifold with almost product structure $F$. If an almost product manifold $\bar{M}$ has a Riemannian metric $g$ such that
\begin{align}\label{2.1}
	g(FX,FY) = g(X,Y),
\end{align}
for any $X, Y \in \Gamma(T\bar{M})$, then $\bar{M}$ is called an almost product Riemannian manifold\cite{yono}, where  $\Gamma(T\bar{M})$ denotes the set of all vector fields of $\bar{M}$. Let $\bar{\nabla}$ denotes the Levi-Civita connection on $\bar{M}$ with respect to the Riemannian metric $g$. if $(\bar{\nabla}_XF)Y = 0$, for any vector $ X, Y \in \Gamma(T\bar{M})$, the $\bar{M}$ is called a locally product Riemannian manifold \cite{li}.\\
Let $M$ be a Riemannian manifold isometrically immersed in $\bar{M}$ and we denote by the symbol $g$ the Riemannian metric induced on $M$. Let $\Gamma(TM)$ denote the Lie algebra of vector fields in $M$ and $\Gamma(T^\perp M)$, the set of all vector fields normal to $M$. If $\nabla$ be the induced Levi-Civita connection on $M$, the Gauss and Weingarten formulas are respectively given by 
\begin{align}\label{2.2}
	\bar{\nabla}_XY = \nabla_XY + \sigma(X,Y),
\end{align}
and 
\begin{align}\label{2.3}
	\bar{\nabla}_XN = -A_NX + \nabla^\perp_XN,
\end{align}
for any $ X,Y \in \Gamma(TM)$ and $N \in \Gamma(T^\perp M)$, where $\nabla^\perp$ is the normal connection on $T^\perp M$ and $A$ the shape operator. The shape operator and the second fundamental form of $M$ are related by
\begin{align}\label{2.4}
	g(A_NX,Y) = g(\sigma(X,Y),N),
\end{align}
for any $ X,Y \in \Gamma(TM)$ and $N \in \Gamma(T^\perp M)$, and $g$ denotes the induced metric on $M$ as well as the metric on $\bar{M}$.\\
For a tangent vector field $X$ and a normal vector field $N$ of $M$, we can write
\begin{align}\label{2.5}
	FX = TX + \omega X , \hspace*{0.5cm} FN = BN +CN,
\end{align}
where $TX$ and $\omega X$ (respectively, $BN \hspace*{0.1cm}\textnormal{and}\hspace*{0.1cm} CN$) are the tangential and normal components of $FX$ (respectively, of $FN$).\\
Moreover, from (\ref{2.1}) and (\ref{2.5}), we have
\begin{align}\label{2.6}
	g(TX,Y) = g(X,TY),
\end{align}
for any $X ,Y \in \Gamma(TM)$.\\
We can now specify the following classes of submanifolds of locally product Riemannian manifolds:\\
(1) A submanifold $M$ of a locally product Riemannian manifold $\bar{M}$ is said to be slant (see \cite{Atcken,chen2005,sahinn}), if for each non-zero vector $X$ tangent to $M$, the angle $\theta(X)$ between $FX$ and $T_pM$ is a constant, i.e., it does not depend on the choice of $ p \in M$ and $ X \in T_pM$.\\
(2) A submanifold $M$ of a locally product Riemannian manifold $\bar{M}$ is called semi-invariant submanifold (see \cite{bejancu,sahin11}) of $\tilde{M}$ if there exists a differentiable distribution $D : p \rightarrow D_p \subset T_pM$ such that $D$ is invariant with respect to $F$ and the complementary distribution $D^\perp$ is anti-invariant with respect to $F$.\\
(3) A submanifold $M$ of a locally product Riemannian manifold $\bar{M}$ is called semi-slant (see \cite{li,papaghuic}), if it is endowed with two orthogonal distributions $D$ and $D^\theta$, where $D$ is invariant with respect to $F$ and $D^\theta$ is slant, i.e., $\theta(X)$ is the angle between $FX$ and $D^\theta_p$ is constant for any $X \in D^\theta_p$ and $p \in M$.
\begin{definition}
	A submanifold $M$ of a locally product Riemannian manifolld $\bar{M}$ is called pointwise slant \cite{Gulbahar}, if at each point $p \in M$, the Wirtinger angle $\theta(X)$ between $FX$ and $T_pM$ is independent of the choice of the non-zero vector $X \in T_pM$. In this case, the Wirtinger angle gives rise to a real valued function $\theta : TM - \{0\} \rightarrow \mathbb{R}$ which is called Wirtinger function or slant function of the pointwise slant function.
\end{definition}
We note that a pointwise slant submanifold of a locally product Riemannian manifold is called slant, in the sense of \cite{Atcken,sahinn}, if its Wirtinger function $\theta$ is globally constant. We also note that every slant submanifold is a pointwise slant slant submanifold.
\par From Chen's result (Lemma 2.1) of \cite{Garay}, we can easily show that $M$ is a pointwise slant submanifold of a locally product Riemannian manifold $\bar{M}$ if and only if 
\begin{align}\label{2.7}
	T^2 = (\cos^2 \theta)I,
\end{align}
for some real-valued function $\theta$ defined on $M$, where $I$ denotes the identity transformation of the tangent bundle $TM$ of $M$. The following relations are the consequences of (\ref{2.7}) as 
\begin{align}\label{2.8}
		g(TX,TY) = (\cos^2\theta)g(X,Y), \hspace{0.5cm} g(\omega X,\omega Y) = (\sin^2\theta)g(X,Y).
\end{align}
Also, for a pointwise bi-slant submanifold of locally product Riemannian manifold, (\ref{2.5}) and (\ref{2.7}) yields
\begin{align}\label{2.9}
	B\omega X = (\sin^2\theta)X \hspace{0.5cm} C\omega X = - \omega TX.
\end{align}
\section{Pointwise bi-slant submanifolds}
In this section, we define and study pointwise bi-slant submanifolds of a locally product Riemannian manifold. 
\begin{definition}
	Let $\bar{M}$ be a locally product Riemannian manifold and $M$ a real submanifold of $\bar{M}$. The we say $M$ is a bi-slant submanifold if there exists a pair of orthogonal distributions $D_1$ and $D_2$ of $M$, at a point $ p \in M$ such that \\
(a) $ TM = D_1 \oplus D_2$;\\
(b) $ JD_1 \perp D_2$ and $JD_2 \perp D_1$;\\
(c) The distributions $D_1, D_2$ are pointwise slant with slant functions $\theta_1 , \theta_2 $, respectively.\\
\par The pair $\{ \theta_1, \theta_2\}$ of slant functions  is called the bi-slant function. A pointwise bi-slant submanifold $M$ is called proper if its bi-slant function satisfies $ \theta_1, \theta_2 \neq 0, \frac{\pi}{2}$ and both $\theta_1, \theta_2$ are not constant on $M$.	
\end{definition}
Note that (\ref{2.5}) and condition $(b)$ in the Definition 3.1 imply that
\begin{align}\label{3.1}
		T(D_i) \subset D_i, \ i = 1,2.
\end{align} 
Given a pointwise bi-slant submanifold $M$ of locally product Riemannian manifold $\bar{M}$, for any $ X \in \Gamma(TM)$, we put
\begin{align}\label{3.2}
	X = P_1X + P_2X 
\end{align}
where $P_i$ is the projection from $\Gamma(TM)$ onto $D_i$. Clearly, $P_iX$ is the components of $X$ in $D_i, i = 1, 2$. In particular, if $X \in D_i$, we have $ X = P_iX$. If we put $T_i = P_i \circ T$, then we can find (\ref{3.2}) that
\begin{align}\label{3.3}
		FX = T_1X + T_2X + \omega X,
\end{align} 
for any $ X \in \Gamma(TM)$.\\
From now onwards, we assume the ambient manifold $\bar{M}$ is locally product Riemannian manifold and $M$ is pointwise bi-slant submanifold in $\bar{M}$.\\
Now we give the following useful lemma for later use.
\begin{lemma}
	Let $M$ be a pointwise bi-slant submanifold of a locally product Riemannian manifold $\bar{M}$ with pointwise slant distributions $D_1$ and $D_2$ with distinct slant functions $\theta_1$ and $\theta_2$, respectively. Then\\
	(i) For $ X, Y \in D_1$ and $Z \in D_2$, we have
	\begin{eqnarray}\label{3.4}
		(\sin^2\theta_2 - \sin^2\theta_1)g(\nabla_XY,Z) &=&\nonumber g\big\{(\sigma(X,Z),\omega T_1Y) + (\sigma(X,T_2Z),\omega Y)\big\} \\ \nonumber && + g\big\{(\sigma(X,Y),\omega T_2Z) + (\sigma(X, T_1Y),\omega Z)\big\}.\\
	\end{eqnarray}
(ii) For $ Z, W \in D_2$ and $ X \in D_1$, we have 
\begin{eqnarray}\label{3.5}
	(\sin^2\theta_1 - \sin^2\theta_2)g(\nabla_ZW,X) &=&\nonumber g\big\{(\sigma(X,Z),\omega T_2W) + (\sigma(Z,T_1X),\omega W)\big\} \\ \nonumber && + g\big\{(\sigma(Z,W),\omega T_1X) + (\sigma(Z, T_2W),\omega X)\big\}.\\
\end{eqnarray}
\end{lemma} 
\begin{proof}
	For $ X, Y \in D_1$ and $ Z \in D_2$, we have 
	\begin{align*}
		g(\nabla_XY,Z) = g(\bar{\nabla}_XY,Z) = g(F\bar{\nabla}_XY,FZ).
	\end{align*}
Using the locally product structure and (\ref{2.5}), we have
\begin{eqnarray*}
	g(\nabla_XY,Z) &=& g(\bar{\nabla}_XT_1Y,FZ) + g(\bar{\nabla}_X\omega Y,T_2Z) + g(\bar{\nabla}_X\omega Y,\omega Z).\\ &=&  g(\bar{\nabla}_XT^2_1Y,Z) + g(\bar{\nabla}_X\omega T_1Y,Z) - g(A_{\omega Y}X,T_2Z) \\ && - g(\bar{\nabla}_X\omega Z, \omega Y).
\end{eqnarray*}
Again using (\ref{2.5}) and (\ref{2.7}), we obtain
\begin{eqnarray*}
   g(\nabla_XY,Z) &=& \cos^2\theta_1 g(\bar{\nabla}_XY,Z) - \sin 2(\theta_1)X(\theta_1)g(Y,Z) - g(A_{\omega T_1Y}X,Z)\\ && - g(A_{\omega Y}X,T_2Z) - g(\bar{\nabla}_X\omega Z,FY) + g(\bar{\nabla}_X\omega Z,T_1Y).
\end{eqnarray*}
By using the orthogonality of two distributions and the symmetry of shape operator, the above equation reduces to
\begin{eqnarray*}
	\sin^2\theta_1g(\nabla_XY,Z) &=& - g(\sigma(X,Z),\omega T_1Y) - g(\sigma(X,T_2Z),\omega Y) \\ && - g(\bar{\nabla}_XB\omega Z,Y)  - g(\bar{\nabla}_XC\omega Z,Y) - g(A_{\omega Z}X,T_1Y).	
\end{eqnarray*}
Thus, from (\ref{2.9}), we obtain
\begin{eqnarray*}
	\sin^2\theta_1g(\nabla_XY,Z) &=& - g(\sigma(X,Z),\omega T_1Y) - g(\sigma(X,T_2Z),\omega Y) \\ && - \sin^2\theta_2 g(\bar{\nabla}_XZ,Y) - \sin 2(\theta_2) X(\theta_2)g(Y,Z) \\ && + g(\bar{\nabla}_X\omega T_2Z,Y) -  g(A_{\omega Z}X,T_1Y).
\end{eqnarray*}
Using (\ref{2.3}) and the orthogonality of vector fields, we have
\begin{eqnarray*}
	\sin^2\theta_1g(\nabla_XY,Z) &=& - g(\sigma(X,Z),\omega T_1Y) - g(\sigma(X,T_2Z),\omega Y) \\ && + \sin^2\theta_2 g(\bar{\nabla}_XY,Z) - g(A_{\omega T_2Z}X,Y) \\&& - g(A_{\omega Z}T_1Y,X).
\end{eqnarray*}
Now, part $(i)$ of the lemma follows from the above Equation by using (\ref{2.4}). In the similar fashion, we can prove part $(ii)$.
\end{proof}
The following corollary is the immediate consequence of the Lemma 1$(i)$
\begin{corrolory}
	Let $M$ be a pointwise semi-slant submanifold of a locally product Riemannian manifold $\bar{M}$. Then,
	\begin{align*}
		\sin^2\theta g(\nabla_XY,Z) = g(\sigma(X,Y),\omega TZ) + g(\sigma(X,FY),\omega Z),
	\end{align*}
	for any $X, Y \in D_1$ and $Z \in D_2$.
\end{corrolory}
	\begin{proof}
		If we put $\theta_1 = 0$ and $\theta_2 = \theta$, a slant function, then the submanifold $M$ of locally product Riemannian manifold $\bar{M}$ becomes pointwise semi-slant submanifold. In this case, the first two terms in the right hand size of (\ref{3.4}) vanish identically. Thus the relation (\ref{3.4}) reduces to 
		\begin{align*}
			\sin^2\theta g(\nabla_XY,Z) = g(\sigma(X,Y),\omega TZ) + g(\sigma(X,FY),\omega Z).
		\end{align*}
	The same result has been proved in \cite{JGP}.
	\end{proof}
\section{Warped product pointwise bi-slant submanifolds of locally product Riemannian manifold}
Let $(M_1, g_1)$ and $(M_2, g_2)$ be two Riemannian manifolds and $f > 0$, be a positive differentiable function on $M_1$. Consider the product	manifold $M_1 \times M_2$ with its canonical projections $\pi : M_1 \times M_2 \rightarrow M_1$ and $\rho : M_1 \times M_2 \rightarrow M_2$. The warped product $M = M_1 \times_f M_2$ is the product manifold $M_1 \times M_2$ equipped with the Riemannian metric $g$ such that\\
\begin{eqnarray*}
	g(X,Y) = g_1(\pi_\ast(X),\pi_\ast(Y)) + (f \circ \pi)^2g_2(\rho_\ast(X),\rho_\ast(Y))
\end{eqnarray*}
for any tangent vector $X, Y \in TM$, where $\ast$ is the symbol for the tangent maps. It was proved in \cite{Neill} that for any $ X \in TM_1$ and $Z \in TM_2$, the following holds
\begin{eqnarray}\label{4.1}
	\nabla_XZ = \nabla_ZX = (Xln f)Z
\end{eqnarray}
where $\nabla$ denotes the Levi-Civita connection of $g$ on $M$. A warped product manifold $ M = M_1 \times_f M_2$ is said to be trivial if the warping function $f$ is constant. If $M = M_1 \times_f M_2$ is a warped product manifold then $M_1$ is totally geodesic and $M_2$ is a totally umbilical (see \cite{Neill,6}).\\
\begin{lemma}
	Let $M_T \times_f M_\theta$ be a warped product pointwise bi-slant submanifold of a locally product Riemannian manifold $\bar{M}$ such that $M_T$ and $M_\theta$ are pointwise slant submanifolds with slant functions $\theta_1$ and $\theta_2$, respectively of $\bar{M}$. Then we have the following
	\begin{align}\label{4.2}
		g(\sigma(X,W), \omega T_2Z) + g(\sigma(X,T_2Z), \omega W) = -(\sin 2\theta_2)X(\theta_2)g(Z,W)
	\end{align}
for any $X \in TM_T$ and $Z,W \in TM_\theta$.
\begin{proof}
	For any $X \in TM_T$ and $ Z, W \in TM_\theta$, we have
	\begin{align}\label{4.3}
		g(\bar{\nabla}_XZ,W) = g(\nabla_XZ,W) = X(ln f)g(Z,W).
	\end{align}
On the other hand, we also have 
\begin{align*}
		g(\bar{\nabla}_XZ,W) = g(F\bar{\nabla}_XZ,FW) = g(\bar{\nabla}_XFZ,FW). 
\end{align*}
Now for any  $X \in TM_T$ and $ Z, W \in TM_\theta$. Using (\ref{2.5}), we obtain
\begin{align*}
		g(\bar{\nabla}_XZ,W) = 	g(\bar{\nabla}_XT_2Z,T_2W) + g(\bar{\nabla}_XT_2Z,\omega W) + g(\bar{\nabla}_X\omega Z,FW).
\end{align*}
Then from (\ref{2.1}), (\ref{2.2}), (\ref{4.1}) and the locally product Riemannian structure, we derive
\begin{small}
	\begin{eqnarray*}
		g(\bar{\nabla}_XZ,W) &=& \cos^2\theta_2 X(ln f)g(Z,W) + g(\sigma(X,T_2Z),\omega W) + g(\bar{\nabla}_XF\omega Z,W)\\ &=& \cos^2\theta_2 X(ln f)g(Z,W) + g(\sigma(X,T_2Z),\omega W) + g(\bar{\nabla}_XB\omega Z,W) \\ &&+ g(\bar{\nabla}_XC\omega Z, W). 
	\end{eqnarray*}
\end{small}
Using (\ref{2.9}), we find 
\begin{eqnarray}\label{4.4}
		g(\bar{\nabla}_XZ,W) &=& \cos^2\theta_2 X(ln f)g(Z,W) + g(\sigma(X,T_2Z),\omega W) \\ \nonumber &&	+ \sin^2\theta_2 g(\bar{\nabla}_XZ,W) + \sin2\theta_2 X(\theta_2)g(Z,W) \\ \nonumber && - g(\bar{\nabla}_X\omega T_2Z,W). 
\end{eqnarray}
Thus the lemma follows from (\ref{4.3}) and (\ref{4.4}) by Using (\ref{2.3}) and (\ref{4.1}).
\end{proof}
\end{lemma}
\begin{lemma}
	Let $M_T \times_f M_\theta$ be a warped product pointwise bi-slant submanifold of a locally product Riemannian manifold $\bar{M}$ such that $M_T$ and $M_\theta$ are pointwise slant submanifolds with slant functions $\theta_1$ and $\theta_2$, respectively of $\bar{M}$. Then we have the following
	\begin{align}\label{4.6}
		g(\sigma(X,Z), \omega W) + g(\sigma(X,W), \omega Z) = -2(\tan \theta_2)X(\theta_2)g(T_2Z,W)
	\end{align}
	for any $X \in TM_T$ and $Z,W \in TM_\theta$.
\end{lemma}
\begin{proof}
	The proof of this lemma follows by Interchanging $Z$ by $T_2Z$ for any $Z \in TM_2$ in (\ref{4.2}) and then by using (\ref{2.7}).
\end{proof}

\begin{lemma}	
	Let $M_T \times_f M_\theta$ be a warped product pointwise bi-slant submanifold of a locally product Riemannian manifold $\bar{M}$ such that $M_T$ and $M_\theta$ are pointwise slant submanifolds with slant functions $\theta_1$ and $\theta_2$, respectively of $\bar{M}$. Then 
	\begin{align}\label{4.6}
	(i) \hspace{0.5 cm}	g(\sigma(X,Z),\omega W) = g(\sigma(X,W),\omega Z),
	\end{align}
\begin{align}\label{4.7}
	(ii) \hspace{0.5cm}g(\sigma(X,Z),\omega Y) = - g(\sigma(X,Y),\omega Z),
\end{align}
for any $X \in TM_T$ and $Z, W \in TM_\theta$.
\end{lemma}
\begin{proof}
	For any $ X \in TM_T$ and $ Z, W \in TM_\theta$, we have
	\begin{eqnarray*}
		g(\sigma(X,Z),\omega W) &=& g(\bar{\nabla}_ZX,\omega W) \\ &=& g(\bar{\nabla}_ZX,FW) - g(\bar{\nabla}_ZX,T_2W).
	\end{eqnarray*}
Using (\ref{2.1}),(\ref{2.5}) and (\ref{4.1}), we obtain
\begin{align*}
		g(\sigma(X,Z),\omega W) = g(\bar{\nabla}_ZT_1X,W) + g(\bar{\nabla}_Z\omega X,W) - X(ln f)g(Z,T_2W).
\end{align*}
On simplification and using (\ref{2.3}), (\ref{2.4}) and (\ref{4.1}), we derive
\begin{align}\label{4.8}
\nonumber	g(\sigma(X,Z),\omega W) = T_1X(ln f)g(Z,W) - g(\sigma(Z,W),\omega X) - X(ln f)g(Z,T_2W).\\
\end{align}
Then from polarization, we get
\begin{align}\label{4.9}
\nonumber	g(\sigma(X,W),\omega Z) = T_1X(ln f)g(Z,W) - g(\sigma(Z,W),\omega X) - X(ln f)g(T_2Z,W).\\
\end{align}
Subtracting (\ref{4.9}) from (\ref{4.8}) and using (\ref{2.6}), we obtain
\begin{align*}
	g(\sigma(X,Z),\omega W) - g(\sigma(X,W),\omega Z) = 0.
\end{align*}
Hence, the proof follows from the above relation.\\
For part $(ii)$, the proof follows same as part $(i)$.
\end{proof}

\begin{theorem}	
	Let $M = M_T \times_f M_\theta$ be a warped product pointwise bi-slant submanifold of a locally product Riemannian manifold $\bar{M}$ such that $M_T$ and $M_\theta$ are pointwise slant submanifolds with distinct slant functions $\theta_1$ and $\theta_2$, respectively of $\bar{M}$. Then we have
	\begin{align}\label{4.10}
	\nonumber	g(A_{\omega T_1X}W + A_{\omega X}T_2W,Z) + g(A_{\omega T_2W}X + A_{\omega W}T_1X,Z)\\ =  (\sin^2\theta_2 - \sin^2\theta_1)X(ln f)g(Z,W).
	\end{align}
	for any $X, Y\in TM_T$ and $Z, W \in TM_\theta$.
\end{theorem}
\begin{proof}
	For any $X, Y\in TM_T$ and $Z, W \in TM_\theta$, we have
	\begin{align}\label{4.11}
		g(\bar{\nabla}_ZX,W) = g(\nabla_ZX,W) = X(ln f)g(Z,W).
	\end{align}
On the other hand, for $X, Y\in TM_T$ and $Z, W \in TM_\theta$, we have
\begin{align*}
	g(\bar{\nabla}_ZX,W) = g(F\bar{\nabla}_ZX,FW) = g(\bar{\nabla}_ZFX,FW).
\end{align*}
Therefore, by using (\ref{2.5}), we get
\begin{align*}
	g(\bar{\nabla}_ZX,W) = g(\bar{\nabla}_ZT_1X,FW) + g(\bar{\nabla}_Z\omega X,T_2W) + g(\bar{\nabla}_Z\omega X,\omega W). 
\end{align*}
Using (\ref{2.1}), (\ref{2.3}) and definition of locally product Riemannian manifold, we obtain
\begin{align*}
	g(\bar{\nabla}_ZX,W) = g(\bar{\nabla}_ZFT_1X,W) - g(A_{\omega X}Z,T_2W) - g(\bar{\nabla}_Z \omega W, \omega X).
\end{align*}
From (\ref{2.5}) and symmetry of shape operator, we derive
\begin{eqnarray*}
g(\bar{\nabla}_ZX,W) &=& g(\bar{\nabla}_ZT_1^2X,W) + g(\bar{\nabla}_Z\omega T_1X,W) - g(A_{\omega X}T_2W,Z) \\ && - g(F\bar{\nabla}_Z\omega W,X) + g(\bar{\nabla}_Z\omega W,T_1X) \\ &=& \cos^2\theta_1g(\bar{\nabla}_ZX,W) - \sin 2\theta_1Z(\theta_1)g(X,W) - g(A_{\omega T_1X}Z,W) \\ && - g(A_{\omega X}T_2W,Z) - g(\bar{\nabla}_ZF\omega W,X) - g(A_{\omega W}Z,T_1X).	
\end{eqnarray*}
Using (\ref{2.2}), (\ref{2.5}), (\ref{4.1}), (\ref{4.11}) and the orthogonality of vector fields and symmetry of shape operator, we get
\begin{eqnarray*}
	\sin^2\theta_1 X(ln f)g(Z,W) &=& -g(A_{\omega T_1X}W + A_{\omega X}T_2W,Z) - g(\bar{\nabla}_ZB\omega W,X) \\ && - g(\bar{\nabla}_ZC\omega W,X) - g(A_{\omega W}T_1X,Z).
\end{eqnarray*}
Using (\ref{2.9}), we arrive at
\begin{eqnarray*}
	\sin^2\theta_1 X(ln f)g(Z,W) &=& -g(A_{\omega T_1X}W + A_{\omega X}T_2W,Z) - \sin^2\theta_2g(\bar{\nabla}_ZW,X) \\ && - \sin 2\theta_2 Z(\theta_2)g(X,W) + g(\bar{\nabla}_Z\omega T_2W,X) \\ && - g(A_{\omega W}T_1X,Z).  
\end{eqnarray*}
Further, using orthogonality of vector fields and the relation (\ref{2.2}), (\ref{2.3}) and (\ref{4.1}), we obtain
\begin{eqnarray*}
	\sin^2\theta_1 X(ln f)g(Z,W) &=& -g(A_{\omega T_1X}W + A_{\omega X}T_2W,Z) \\&& + \sin^2\theta_2 X(ln f)g(Z,W) - g(A_{\omega T_2W}Z,X) \\ && -  g(A_{\omega W}T_1X,Z).
\end{eqnarray*}
Again using the symmetry of shape operator, we obtain (\ref{4.10}) from the above relation. Hence the proof is complete.
\end{proof}
\section{Characterization for warped product pointwise bi-slant submanifolds of locally product Riemannian manifolds.}
In this section, we will prove the characterization for warped product pointwise bi-slant submanifolds of locally product Riemannian manifolds. For this, we need the following well known theorem of Hiepko's.
\begin{theorem}\cite{Hiepko}
	Let $D_1$ and $D_2$ be two orthogonal distribution on a Riemannian manifold $M$. Suppose that $D_1$ and $D_2$ both are involutive such that $D_1$ is totally geodesic folation and $D_2$ is a spherical foliation. Then $M$ is locally isometric to a non-trial warped product $M_1 \times_fM_2$, where, $M_1$ and $M_2$ are integral manifolds of $D_1$ and $D_2$, respectively.
\end{theorem}
The following result provides a characterization for warped product pointwise bi-slant submanifolds of locally product Riemannian manifolds.
\begin{theorem}
		Let $M$ be a proper pointwise bi-slant submanifold of a locally product Riemannian manifold $\bar{M}$ with pointwise slant distributions $D_1$ and $D_2$. Then $M$ is locally a warped product pointwise bi-slant submanifold of the form $M_T \times_f M_\theta$, where $M_T$ and $M_{\theta}$ are pointwise slant submanifolds with distinct slant functions $\theta_1$ and $\theta_2$, respectively of $\bar{M}$ if and only if the shape operator of $M$ satisfies
		\begin{align}\label{5.1}
			A_{\omega T_1X}Z + A_{\omega X}T_2Z + A_{\omega T_2Z}X + A_{\omega Z}T_1X = (\sin^2\theta_2-\sin^2\theta_1)X(\mu)Z
		\end{align}
	for $X \in D_1$, $Z \in D_2$ and for a smooth function $\mu$ on $M$ satisfying $W(\mu) = 0$ for any $W \in D_2$.
\end{theorem}
\begin{proof}
	Let  $M = M_T \times_f M_\theta$ be a warped product pointwise bi-slant submanifold of a locally product Riemannian manifold $\bar{M}$. Then from Lemma 4.2$(ii)$, we have
	\begin{align}\label{5.2}
		g(A_{\omega Y}Z + A_{\omega Z}Y,X) = 0
	\end{align}
for any $X, Y \in TM_1$ and $Z \in TM_2$. Interchanging $Y$ by $T_1Y$ in (\ref{5.2}), we obtain
\begin{align}\label{5.3}
	g(A_{\omega T_1Y}Z + A_{\omega Z}T_1Y,X) = 0.
\end{align}
Again interchanging $Z$ by $T_2Z$ in (\ref{5.2}), we obtain
\begin{align}\label{5.4}
		g(A_{\omega Y}T_2Z + A_{\omega T_2Z}Y,X) = 0.
\end{align}
Adding equations (\ref{5.3}) and (\ref{5.4}), we get
\begin{align}
	g(A_{\omega T_1Y}Z + A_{\omega Z}T_1Y + A_{\omega Y}T_2Z + A_{\omega T_2Z}Y,X) = 0.
\end{align}
Then (\ref{5.1}) follows from (\ref{4.10}) by using the above fact.\\
\par Conversely, if $M$ be a pointwise bi-slant subamnifold of a locally product Riemannian manifold with pointwise slant distributions $D_1$ and $D_2$ such that (\ref{5.1}) holds, then from Lemma 3.2$(i)$, we have
\begin{eqnarray*}
	(\sin^2\theta_2 - \sin^2\theta_1)g(\nabla_XY,Z) &=& g(A_{\omega T_1Y}Z + A_{\omega Z}T_1Y\\ &&  +  A_{\omega Y}T_2Z + A_{\omega T_2Z}Y,X)
\end{eqnarray*}
for any $ X,Y \in D_1$ and $Z \in D_2$. Using the above condition (\ref{5.1}), we have
\begin{align*}
	g(\nabla_XY,Z) = X(\mu)g(X,Z) = 0
\end{align*}
which indicates that the leaves of the distributions are totally geodesic in $M$.\\
On the other hand, from Lemma 3.2$(ii)$, we have 
\begin{eqnarray*}
	(\sin^2\theta_1 - \sin^2\theta_2)g(\nabla_ZW,X)&=&g(A_{\omega T_2W}X + A_{\omega W}T_1X \\ && + A_{\omega T_1X}W + A_{\omega X}T_2W,Z). 
\end{eqnarray*} 
From the hypothesis of the theorem i.e., (\ref{5.1}), we get
\begin{align}\label{5.6}
	g(\nabla_ZW,X) = -X(\mu)g(Z,W).
\end{align}
By polarization, we arrive at
\begin{align}\label{5.7}
	g(\nabla_WZ,X) = -X(\mu)g(Z,W).
\end{align}
On subtracting (\ref{5.7}) from (\ref{5.6}) and by the definition of Lie bracket, we obtain $g([Z,W],X) = 0$, which depicts that the distribution $D_2$ is integrable. If we consider a leaf $M_2$ of $D_2$ and the second fundamental form $\sigma_2$ of $M_2$ in $M$, then from (\ref{5.6}), we have
\begin{align*}
	g(\sigma_2(Z,W),X) = g(\nabla_ZW,X) = -X(\mu)g(Z,W).
\end{align*}
Now, by the definition of the gradient we have $\sigma_2(Z,W) = -\bar{\nabla}_{\mu}g(Z,W)$, such that $\bar{\nabla}_{\mu}$ is the gradient of $\mu$. The above relations shows that the leaf of $M_2$ is totally Umbilical in $M$ with the mean curvature vector $H_2 = -\bar{\nabla}_{\mu}$. Since $W(\mu) = 0$ for any $W \in D_2$, which clearly shows that the mean curvature is parallel. Thus, the spherical condition is satisfied. Then by Hiepko's Theorem $M$ is locally a warped product pointwise bi-slant submanifold. Hence the proof is complete.
\end{proof}
\par The following immediate consequences of the above theorem are given below:\\\\
 1. In Theorem 5.2, if $\theta_1=0$ and $\theta_2 = \theta$, a slant function, then the submanifold $M$ of locally product Riemannian manifold $\bar{M}$ becomes a pointwise semi-slant submanifold which has been studied in \cite{JGP}. In this case, the first two terms in the left hand side of (\ref{5.1}) vanish identically. Thus, the relation (\ref{5.1}) is true for pointwise semi-slant warped product and it reduces to
\begin{align*}
	A_{\omega TZ}X + A_{\omega Z}FX = (\sin^2\theta)X(\mu)Z
\end{align*}
for $X \in D_1$ and $Z \in D_2$, where $D_1$ and $D_2$ are complex and proper pointwise slant distributions of $M$. The same has been proved in \cite{JGP}.\\\\
 2. In Theorem 5.2, if we consider $\theta_1 = \theta$ a constant slant angle and $\theta_2 = \frac{\pi}{2}$, then it is a case of hemi-slant warped products. In this case, the second and third term in the left hand side of (\ref{5.1}) vanish identically. Thus the relation (\ref{5.1}) is true for hemi-slant warped products and it reduces to 
\begin{align*}
	A_{\omega TX}Z + A_{FZ}TX = (\cos^2\theta)X(\mu)Z
\end{align*}
for $X \in D_{\theta}$ and $Z \in D^\perp$, where $ D_{\theta}$ and $ D^\perp$ are proper slant and totally real distributions.\\\\
 3. In Theorem 5.2, if $\theta_1 = 0$ and $\theta_2 = \frac{\pi}{2}$, then it is a case of CR-warped product. In this case all the terms in the left hand side of(\ref{5.1}) vanish identically. Thus the relation (\ref{5.1}) is true for CR-warped products and it will be 
\begin{align*}
	A_{FZ}FX = X(\mu)Z
\end{align*}
for $X \in D$ and $Z \in D^{\perp}$, where $D$ and $D^{\perp}$ are complex and totally real distributions of $M$.
\section{Some examples on warped product pointwise bi-slant submanifolds of locally product Riemannian manifold.}
\begin{example}
	Let $\mathbb{R}^4$ be the Euclidean space with the cartesian coordinates given by $(x_1,x_2,y_1,y_2)$ and the almost product structure 
	\begin{align*}
		F\bigg(\frac{\partial}{\partial x_i}\bigg) = \frac{\partial}{\partial x_i}, \hspace*{0.5cm} 	F\bigg(\frac{\partial}{\partial y_j}\bigg) = -\frac{\partial}{\partial y_j}, 1\leq i,j \leq 2. 
	\end{align*}
A submanifold $M$ of $\mathbb{R}^4$ defined by
\begin{align*}
\chi (u,v,w) = (wu\cos v, wu\sin v, w\cos v,w\sin v).	
\end{align*}
It is easy to see that the tangent bundle $TM$ of $M$ is spanned by the following vectors
\begin{eqnarray*}
	v_1 = w\cos v \frac{\partial}{\partial x_1} + w\sin v \frac{\partial}{\partial x_2},
\end{eqnarray*}
\begin{eqnarray*}
	v_2 = -wu\sin v \frac{\partial}{\partial x_1} + wu\cos v \frac{\partial}{\partial x_2} - w\sin v \frac{\partial}{\partial y_1} + w\cos v \frac{\partial}{\partial y_2},
\end{eqnarray*}
\begin{eqnarray*}
	v_3 = u\cos v \frac{\partial}{\partial x_1} + u\sin v \frac{\partial}{\partial x_2} + \cos v \frac{\partial}{\partial y_1} + \sin v \frac{\partial}{\partial y_2}.
\end{eqnarray*}
Then , clearly we obtain
\begin{eqnarray*}
	Fv_1 = w\cos v \frac{\partial}{\partial x_1} + w\sin v \frac{\partial}{\partial x_2},
\end{eqnarray*}
\begin{eqnarray*}
	Fv_2 = -wu\sin v \frac{\partial}{\partial x_1} + wu\cos v \frac{\partial}{\partial x_2} + w\sin v \frac{\partial}{\partial y_1} - w\cos v \frac{\partial}{\partial y_2},
\end{eqnarray*}
\begin{eqnarray*}
	Fv_3 = u\cos v \frac{\partial}{\partial x_1} + u\sin v \frac{\partial}{\partial x_2} - \cos v \frac{\partial}{\partial y_1} - \sin v \frac{\partial}{\partial y_2}.
\end{eqnarray*}
Then, we find that $D_1 = span\{v_1,v_3\}$ is a proper pointwise slant distribution with slant angle $\theta_1 = \cos^{-1}\big(\frac{u}{\sqrt{1+u^2}}\big)$ and $D_2 = span \{v_2\}$ is again a proper pointwise slant distribution with slant angle $\theta_2 = \cos^{-1}\big(\frac{u^2-1}{u^2 +1}\big)$. Thus, $M$ is a proper pointwise bi-slant submanifold of $\mathbb{R}^4$.\\
\par It is easy to verify that both the distributions $D_1$ and $D_2$ are integrable. If we denote the integrable manifolds of $D_1$ and $D_2$ by $M_T$ and $M_{\theta}$, respectively. Then the metric tensor $g$ of product manifold $M$ is given by 
\begin{align*}
	g = g_{M_T} + w^2(1 + u^2)g_{M_\theta},
\end{align*}
where,
\begin{align*}
	g_{M_T} = w^2 du^2 + (1 + u^2) dw^2 \hspace{0.5cm} and \hspace{0.5cm} g_{M_\theta} = dv^2.
\end{align*}
Hence, $M$ is a proper non-trival warped product pointwise bi-slant submanifold of $\mathbb{R}^4$ with warping function $ f = \sqrt{w^2(1+u^2)}$ and whose bi-slant angles $\theta_1,\theta_2 \neq 0, \frac{\pi}{2}$.

\end{example}
\begin{example}
	Let $\mathbb{R}^6 = \mathbb{R}^3 \times \mathbb{R}^3$ be a locally product Riemannian manifold with cartesian coordinates $(x_1, x_2, x_3, y_1, y_3,y_3)$. Consider a submanifold $M$ of $\mathbb{R}^6$ defined by 
	\begin{align*}
		\chi(u,v,w) = (v\cos u, v\sin u, -v+w, w\cos u, w\sin u, v+w),
	\end{align*}
	with almost product structure $F$ defined by 
	\begin{align*}
		F\bigg(\frac{\partial}{\partial x_i}\bigg) = - \bigg(\frac{\partial}{\partial x_i}\bigg), \hspace{0.5 cm} F\bigg(\frac{\partial}{\partial y_j}\bigg) =  \bigg(\frac{\partial}{\partial y_j}\bigg), 1 \leq i,j \leq 3.
	\end{align*}
	It is easy to see that its tangent space $TM$ of $M$ is spanned by the following vectors
	\begin{align*}
		v_1 = -v\sin u \frac{\partial}{\partial x_1} + v\cos u \frac{\partial}{\partial x_2} - w\sin u\frac{\partial}{\partial y_1} + w\cos u \frac{\partial}{\partial y_2},
	\end{align*}
	\begin{align*}
	v_2 = \cos u \frac{\partial}{\partial x_1} + \sin u  \frac{\partial}{\partial x_2} - \frac{\partial}{\partial x_3} + \frac{\partial}{\partial y_3},
\end{align*}
\begin{align*}
	v_3 = \frac{\partial}{\partial x_3} + \cos u \frac{\partial}{\partial y_1} + \sin u \frac{\partial}{\partial y_2} + \frac{\partial}{\partial y_3}.
\end{align*}
Then, we have 
	\begin{align*}
	Fv_1 = v\sin u \frac{\partial}{\partial x_1} - v \cos u \frac{\partial}{\partial x_2} - w \sin u \frac{\partial}{\partial y_1} + w \cos u  \frac{\partial}{\partial y_2},
\end{align*}
\begin{align*}
	Fv_2 = -\cos u \frac{\partial}{\partial x_1} - \sin u  \frac{\partial}{\partial x_2} +  \frac{\partial}{\partial x_3} +  \frac{\partial}{\partial y_3},
\end{align*}
\begin{align*}
	Fv_3 = -\frac{\partial}{\partial x_3} + \cos u \frac{\partial}{\partial y_1} + \sin u  \frac{\partial}{\partial y_2 } + \frac{\partial}{\partial y_3}.
\end{align*}
Let us put $D_1 = span \{v_1\}$ is a proper slant distribution with slant angle $\theta_1 = \cos^{-1}\big(\frac{w^2-v^2}{w^2 +v^2}\big)$ and $D_2 = span \{v_2,v_3\}$ is again a proper slant distribution with slant angle $\theta_2 = \cos^{-1}\big(\frac{2}{3}\big)$. Hence the submanifold $M$ defined by $\chi$ is a bi-slant submanifold.
\par It is easy to verify that both the distributions $D_1$ and $D_2$ are integrable. If we denote the integrable manifolds of $D_1$ and $D_2$ by $M_T$ and $M_{\theta}$, respectively. Then the metric tensor $g$ of product manifold $M$ is given by 
\begin{align*}
	g = g_{M_{\theta}} + (v^2 + w^2)g_{M_T}
\end{align*}
where
\begin{align*}
	g_{M_\theta} = 3(dv^2+dw^2) \hspace{0.5cm} and \hspace{0.5cm} g_{M_T} = du^2.
\end{align*}
Hence, $M$ is a proper non-trival warped product bi-slant submanifold of $\mathbb{R}^6$ with warping function $ f = \sqrt{v^2 + w^2}$ and whose bi-slant angles $\theta_1,\theta_2 \neq 0, \frac{\pi}{2}$.
 \end{example}

\textbf{Conflicts of Interest}: The authors declare no conflict of interest.

\end{document}